\documentclass[11pt,reqno]{amsart}\usepackage[letterpaper, portrait, margin=.9in, headheight=0pt]{geometry}
\usepackage{mathtools}
\usepackage{enumerate}
\usepackage{amsmath,amsthm,amssymb,amsfonts}
\usepackage{stmaryrd}
\usepackage{hyperref}
\usepackage[normalem]{ulem}
\usepackage{soul}
\usepackage{xcolor}
\usepackage{tikz}
\usepackage{tikz-cd}
\usepackage{float}
\usetikzlibrary{calc, positioning, fit, shapes.misc,arrows,bending, decorations.pathmorphing}
\usepackage{subcaption}
\usepackage{blkarray}
\usepackage{cleveref}
\usepackage{transparent}
\usepackage{graphicx}
\newtheorem{thm}{Theorem} 
\theoremstyle{definition}
\newtheorem{exam}[thm]{Example}
\newtheorem{defi}[thm]{Definition}
\newtheorem{lem}[thm]{Lemma}

\newtheorem{conj}[thm]{Conjecture}
\newtheorem{prop}[thm]{Proposition}
\newtheorem{ques}[thm]{Question}

\usepackage{titlesec}

\titleformat{\section}
       {\normalfont
       \fontsize{14}{17}\bfseries}{\thesection}{1em}{}

\titleformat{\subsection}
       {\normalfont\fontsize{12}{17}\bfseries}{\thesubsection}{1em}{}

\title{Proper Additive Edge Colorings of Regular Graphs}

\begin{document}


\maketitle

\begin{center}
IAN GOSSETT\\

\end{center}
\vspace{.2cm}

\begin{abstract} 

We show that if $G$ is a $d$-regular Vizing-class-1 graph, then the proper additive chromatic index of $G$, denoted $\eta'_p(G)$, is equal to its chromatic index. This verifies that a strengthening of the Additive Coloring Conjecture of Czerwi\'{n}ski et al. holds for line graphs of $d$-regular Vizing-class-1 graphs. We show that if $G$ is a $d$-regular Vizing-class-2 graph, $\eta'_{p}(G)\leq \frac{(2^{\lceil \log_2 (d+1)\rceil})^2+2}{3}$, and if $G$ is a $d$-regular Vizing-class-2 graph that admits a proper edge coloring with a smallest color class of size $r$ and $\text{girth}(G)\geq 6r-5$, then $\eta_p'(G)\leq 2d$, among other results. 
        
\end{abstract}

\section{Introduction}\label{sec:intro}

All graphs considered in this article are assumed to be finite, have no self-adjacent vertices, and no parallel edges. Let $\mathbb{N}$ denote the set of positive integers. An \textit{additive coloring} of a graph $G$, first studied in  \cite{MR2552893} (and originally called a \textit{lucky labeling}), is a function $c:V(G)\rightarrow \mathbb{N}$ such that whenever $v$ and $w$ are adjacent vertices in $G$, $\sum_{u\in N(v)}c(u)\neq \sum_{u \in N(w)}c(u)$.  If, in addition, $c$ is also a proper coloring of $G$, then $c$ is called a \textit{proper additive coloring} of $G$. The \textit{additive chromatic number} of a graph $G$, denoted $\eta (G)$, is the least integer $k$ such that the vertices of $G$ can be additively colored using labels from the set $\{1,2,\ldots,k\}$, and the \textit{proper additive chromatic number} of $G$, denoted, $\eta_p(G)$ is the least $k$ such that there exists a proper additive coloring of $G$ using labels from the set $\{1,2,\dots,k\}$.

The edge coloring analog of an additive coloring is called an \textit{additive edge coloring}. An additive edge coloring is a function $c:E(G)\rightarrow \mathbb{N}$ such that for any two incident edges $e_1$ and $e_2$,  $\sum_{e\in N'(e_1)}c(e)\neq \sum_{e\in N'(e_2)}c(e)$. (Here, $N'(e)$ denotes the set of edges incident to $e$ in $G$.) Equivalently, one can define an additive edge coloring of $G$ to be an additive coloring of its line graph, $L(G)$. A \textit{proper additive edge coloring} is an additive edge coloring that is also a proper edge coloring. The \textit{additive chromatic index} of $G$, denoted $\eta'(G)$, is the least $k$ such that there exists an additive edge coloring of $G$ using labels from the set $\{1,2,\dots,k\}$.  The \textit{proper additive chromatic index} of $G$, denoted, $\eta'_p(G)$ is the least $k$ such that there exists a proper additive edge coloring of $G$ using labels from the set $\{1,2,\dots,k\}$.

Letting $\chi(G)$ and $\chi'(G)$ denote the traditional chromatic number and chromatic index, respectively, the following inequalities are immediate from the above definitions:
$\eta(G)\leq \eta_p(G),$
$\chi(G)\leq \eta_p(G),$
$\chi'(G)\leq \eta_p '(G),$
$\eta'(G)\leq \eta_p '(G).$ And in \cite{MR2552893}, Czerwiński et al. made the following well-known conjecture regarding additive colorings.

\begin{conj}\label{conj:ACC} \cite{MR2552893} For all graphs $G$, $\eta(G)\leq \chi (G)$.
\end{conj}

Conjecture \ref{conj:ACC} is known as the \textit{Additive Coloring Conjecture}. This conjecture remains open, but some progress has been made -- for example, it was shown in \cite{MR3248491} that if $G$ is planar, $\eta(G)\leq 468$. (An improvement from $\eta(G)\leq 100 280 245 065$ in \cite{MR2552893}, but still quite far from $\eta(G)\leq \chi(G)\leq 4$, which must hold if the additive coloring conjecture is true.)  Further significant results regarding the additive chromatic number of planar graphs can be found in  \cite{MR2876224},\cite{MR3248491},\cite{MR4089840}, and \cite{MR4462999}. Progress for some other specific types of graphs can be found in \cite{MR4074150}. 

There is still plenty of room for progress on the Additive Coloring Conjecture. For example, it appears to be unknown whether there is a constant $k$ such that for all graphs $G$, $\eta(G)\leq k\chi(G)$, or even whether there is a function $f:\mathbb{N}\rightarrow \mathbb{R}$ such that for all graphs $G$, $\eta(G)\leq f(\chi(G))$.  

In \cite{MR2729020}, a similar but different notion to the additive chromatic number called the \text{sigma chromatic number} was introduced, and it is important not to get the two confused. It is also worth noting that some researchers have focused on list-coloring generalizations of additive colorings, e.g. \cite{MR4089840},\cite{MR4748249}, and \cite{MR4462999}, and methods that are applied to traditional list-colorings have been used with only partial success.

In this article, we focus on proper additive edge colorings of $d$-regular graphs. By a well-known theorem of Vizing \cite{MR0180505},  if $G$ is a $d$-regular graph, then $\chi ' (G)=d$ or $\chi ' (G) = d +1$. A $d$-regular graph $G$ is called \textit{Vizing-class-1} if $\chi ' (G)=d$ and is called \textit{Vizing-class-2} graph if $\chi ' (G) =d+1$. In Theorem \ref{thm:class1add}, we show that $\eta_p'(G)\leq \chi'(G)$ whenever $G$ is a $d$-regular Vizing-class-1 graph, thereby verifying the Additive Coloring Conjecture holds for line graphs of $d$-regular Vizing-class-1 graphs. In Theorem \ref{thm:addbound}, we show that if $G$ is a $d$-regular Vizing-class-2 graph, then $\eta'_{p}(G)\leq \frac{(2^{\lceil \log_2 (d+1)\rceil})^2+2}{3}$. In Theorem \ref{thm:girth}, we show that if $G$ is a Vizing-class-2 graph that admits a proper edge coloring with a smallest color class of size $r$ and girth at least $6r-5$, then $\eta_p'(G)\leq 2d$. Other related bounds for classes of $d$-regular Vizing-class-2 graphs are also proven in Section \ref{sec:results}. Throughout this work, we assume familiarity with standard graph theory notation and concepts that can be found in many introductory graph theory texts (e.g.,\cite{MR4874150}, \cite{west_introduction_2000}).

\section{Additive Edge Colorings of $\mathbf{d}$-Regular Graphs}\label{sec:results} 

\begin{thm}\label{thm:class1add} If $G$ is a $d$-regular Vizing-class-1 graph, then every proper edge coloring of $G$ is an additive edge coloring of $G$. Therefore, $\eta_{p}'(G)=\chi ' (G)=d$, and the Additive Coloring Conjecture holds for line graphs of $d$-regular Vizing-class-1 graphs. 
\end{thm}

\begin{proof} Suppose that $c:E(G)\rightarrow \{1,2,\ldots,d\}$ is a proper edge coloring of $G$. For each $v\in V(G)$, define $\Gamma(v)=\{e\in E(G):v \text{ is an endpoint of } e\}$. Note that for each $vw\in E(G)$, we have 

\begin{align*}\sum _{e\in N'(vw)}c(e)&=\left(\sum _{e\in \Gamma(v)}c(e)-c(vw)\right)+\left(\sum_{e\in \Gamma(w)}c(e)-c(vw)\right)\\
&=\sum _{e\in \Gamma(v)}c(e)+\sum_{e\in \Gamma(w)}c(e)-2c(vw)\\
\end{align*}

Thus, if $vw$, $vx$ are incident edges of $G$, then 

\begin{align*}&\sum _{e\in N'(vw)}c(e)-\sum_{e\in N'(vx)}c(e)\\
=& \left( \sum _{e\in \Gamma(v)}c(e)+\sum_{e\in \Gamma(w)}c(e)-2c(vw)\right)\\
& \hspace{3cm}-\left( \sum _{e\in \Gamma(v)}c(e)+\sum_{e\in \Gamma(x)}c(e)-2c(vx)\right)\\
&=\sum_{e\in \Gamma(w)}c(e)-\sum_{e\in \Gamma(x)}c(e)-(2c(vw)-2c(vx))\\
&=\sum_{i=1}^{d}i-\sum_{i=1}^{d}i-2(c(vw)-c(vx))\\
&=-2(c(vw)-c(vx))\neq 0,
\end{align*}

where the final inequality holds because $c$ is a proper edge coloring. Hence, $c$ is an additive edge coloring of $G$, and $\eta_p'(G)= d$. 
\end{proof}

For $d$-regular Vizing-class-2 graphs, determining $\eta'_p(G)$ seems to be more difficult. The remainder of this work is dedicated to proving upper bounds for $\eta'_{p}(G)$ for $d$-regular Vizing-class-2 graphs.

\begin{defi} For $A\subseteq \mathbb{Q}$, define $A^-=\{x-y:x,y\in A \text{ and } x\neq y\}$, and for each $q\in \mathbb{Q}$, define $qA=\{qz:z\in A\}$, and define $A+q=\{z+q:z\in A\}$.\end{defi}

\begin{prop} \label{prop:Acap2A}Suppose that $G$ is a $d$-regular graph. If $A$ is a set of $d+1$ positive integers such that $A^-\cap 2A^-=\emptyset,$ then every proper edge coloring $c:E(G)\rightarrow A$ is an additive edge coloring of $G$.\end{prop}

\begin{proof} Let $c:E(G)\rightarrow A$ be a proper edge coloring of $G$ and for each $v\in V$, define $a_v$ to be the unique element of $A$ that does not appear in $c(\Gamma(v))$.  Similarly to the proof of \ref{thm:class1add}, if $vw$ and $vx$ are incident edges in $G$, we have

\begin{align*}
&\sum _{e\in N'(vw)}c(e)-\sum_{e\in N'(vx)}c(e)\\
=&\sum_{e\in \Gamma(w)}c(e)-\sum_{e\in \Gamma(x)}c(e)-2(c(vw)-c(vx))\\
=&\left(\sum_{a\in A}a-a_w\right)-\left(\sum_{a\in A}a-a_x\right)-2(c(vw)-c(vx))\\
=& (a_x-a_w)-2(c(vw)-c(vx)).\\
\end{align*}

\noindent If $a_w=a_x$, then  \((a_x-a_w)- 2(c(vw)-c(vx))=- 2(c(vw)-c(vx))\neq 0\). If $a_w\neq a_x$, then $(a_x-a_w)\in A^{-}$ and $2(c(vw)-c(vx))\in 2A^{-}$, and since 
$A^-\cap 2A^-=\emptyset$,  \((a_x-a_w)- 2(c(vw)-c(vx))\neq 0\). Hence $c$ is an additive edge coloring of $G$. 
\end{proof}

\begin{thm} \label{thm:addbound}Suppose that $G$ is a $d$-regular Vizing-class-2 graph. Then \[\eta'_{p}(G)\leq \frac{(2^{\lceil \log_2 (d+1)\rceil})^2+2}{3}=\frac{(2^{\lceil \log_2 \chi'(G)\rceil})^2+2}{3},\] where $\lceil x \rceil$ denotes the least integer $\geq x$. In particular, if $d+1=2^n$ for some $n>0$, then \[\eta'_{p}(G)\leq \frac{(d+1)^2+2}{3}=\frac{(\chi'(G))^2+2}{3}.\]\end{thm}

We will make use of the following lemma in the proof of Theorem \ref{thm:addbound}.
\begin{lem} \label{prop:lemma}Define $A_1=\{1,2\}$ and $A_{n+1}=A_n\cup (A_n+3\max(A_n)-2) \quad \text{for } n\geq 2$. For each $n\in \mathbb{N}$, $|A_n|=2^n$, $\max(A_n)=\frac{4^n+2}{3}$, and $A_n^{-}\cap 2A_n^{-}=\emptyset$.\end{lem}

\begin{proof} That $|A_n|=2^n$ for each $n\in \mathbb{N}$ is clear, since $|A_1|=2$, and for each $n\geq 2$ we construct $A_n$ by taking the union of two disjoint sets of size $|A_{n-1}|$, thus acheiving $|A_n|=2|A_{n-1}|.$  We prove the other two parts by induction on $n$. 

We first show that, for each $n$, $\max(A_n)=\frac{4^n+2}{3}$. Note that $\max (A_1)=2=\frac{4^1+2}{3}$. Now, if $n\in \mathbb{N}$ and $\max(A_n)=\frac{4^n+2}{3}$, we have that \begin{align*}
\max(A_{n+1})&=\max(A_n)+3\max(A_n)-2\\
&=\frac{4^n+2}{3}+3\frac{4^n+2}{3}-2\\
&= \frac{4^{n+1}+2}{3},\\
\end{align*}
as needed to complete the inductive step. 

We now show that $A_n^{-}\cap 2A_n^{-}=\emptyset$ for each $n$. First, observe that $A_1^{-}=\{\pm 1\}$ and $2A_1^{-}=\{\pm 2\}$, so  $A_1^{-}\cap 2A_1^{-}=\emptyset$, and the base case holds. Now, fix $n\in \mathbb{N}$ and suppose $A_n^{-}\cap 2A_n^{-}=\emptyset$. 

Write
 \begin{align*}S_1 &= \{x-y: x\in A_n \text{ and } y\in (A_n+3\max(A_n)-2)\},\\
 S_2&= \{ x-y:y \in A_n \text{ and } x\in (A_n+3\max(A_n)-2)\},
 \end{align*}

and let $S=S_1\cup S_2$.

Then  \begin{align*}A^{-}_{n+1} &= \{x-y:x,y\in A_n \text{ and } x\neq y\} \cup \{x-y:x,y\in (A_n+3\max(A_n)-2)\text{ and } x\neq y\} \cup S\\
&= A^{-}_n\cup A^{-}_n\cup S \text{ (shifting $A_n$ does not change its set of differences)}\\
&= A^{-}_n\cup S.\\
\end{align*}

Suppose $w\in A_{n+1}^{-}.$ If $w\in S$, then \begin{align*}|w|&\geq\min(A_n+3\max(A_n)-2))-\max(A_n)\\
&=(1+3\max(A_n)-2)-\max(A_n)\\
&=2\max(A_n)-1.
\end{align*}Hence, \[|2w|\geq4\max(A_n)-2=\max(A_{n+1})>\max(A_{n+1})-1=\max\{|x|:x\in A^{-}_{n+1}\},\] 
and therefore, $2w\notin A^{-}_{n+1}$.

If $w\in A^{-}_n$, then since $\max\{|x|:x\in A_n^{-})\}=\max(A_n)-1$, we have \[|2w|\leq 2(\max(A_n)-1)<2\max(A_n)-1,\] and we saw above that any element of $S$ has magnitude at least $2\max(A_n)-1$, so $2w\notin S$. By the inductive hypothesis, we have that $2w\notin A^{-}_{n}$, so $2w\notin A^{-}_{n+1}=A_n^{-}\cup S$ in this case either. Therefore, if $w\in A_{n+1}^-$, then $2w\notin A_{n+1}^-$ and we conclude that $A^{-}_{n+1}\cap 2A^{-}_{n+1}=\emptyset$, as needed to complete the inductive step.

\end{proof}

\begin{proof}[Proof of Theorem \ref{thm:addbound}]
    By Proposition \ref{prop:Acap2A} and Lemma \ref{prop:lemma}, any proper edge coloring of $G$ with the first $d+1$ elements of $A_n$ as defined in Lemma \ref{prop:lemma}, where $n=\lceil\log_2(d+1)\rceil$ gives an additive edge coloring of $G$. Since \[\max(A_n)=\frac{4^n+2}{3}=\frac{(2^n)^2+2}{3}=\frac{(2^{\lceil \log_2 (d+1)\rceil})^2+2}{3},\] the claim follows. 
\end{proof}

For the special case where $H$ is the line graph of a $d$-regular graph, the bound given in \Cref{thm:addbound} is better than the following known bounds for arbitrary graphs $H$: $\eta(H)\leq k\Delta(H) +1$, where $k$ is the degeneracy of $H$ and $\Delta(H)$ is the maximum degree of $H$, proved in \cite{MR3478612}, and $\eta_p(H)\leq (\Delta(H)+1)^{\chi(H)-1}$, proved in \cite{MR2729020}. We now turn our attention to cases where we can further improve the bound for $\eta_p'(G)$.

 \begin{prop} \label{thm:oddzero} Suppose that $G$ is a $d$-regular Vizing-class-2 graph, and $c:E(G)\rightarrow \{1,2,\dots, d+1\}$ is a proper $(d+1)$-edge-coloring of $G$. For each $v\in V$, let $a_v$ denote the unique element of $\{1,2,\dots,d+1\}$ that does not appear at $v$ in $c$. If, for any path $v_1,v_2,v_3$ of length $2$ in $G$, either $a_{v_1}-a_{v_3}=0$ or $a_{v_1}-a_{v_3}$ is odd, then $c$ is an additive coloring.   Thus, if $G$ admits such a coloring, $\eta'_p(G)=d+1$.
\end{prop}\newpage 

\begin{proof} From the proof of Proposition \ref{prop:Acap2A}, if $vw$ and $vx$ are incident edges in $G$, we have

\begin{align*}
\sum _{e\in N'(vw)}c(e)-\sum_{e\in N'(vx)}c(e)
=& a_x-a_w-2(c(vw)-c(vx)).\\
\end{align*}

Since $x,v,w$ is a path of length $2$ between $x$ and $w$, we have that $a_{x}-a_{w}=0$ or $a_{x}-a_{w}$ is odd. In either case, $a_{x}-a_{w}-2(c(vw)-c(vx))\neq 0$, so $c$ is an additive coloring of $G$.  \end{proof}

\begin{defi}Let $e,e'\in E(G)$. We say that $e'$ is \textit{2-reachable} from $e$ if there exists a path of length two from an endpoint of $e$ to an endpoint of $e'$. 
\end{defi}

Note that if $e$ and $e'$ share an endpoint, then $e'$ is $2$-reachable from $e$. Note also that an edge is $2$-reachable from itself only if it is contained in some $3$-cycle of $G$. 

\begin{defi} Let $c:E(G)\rightarrow \{1,2,\dots,d+1\}$ be a proper edge coloring of a graph $G$. Call $c$ a \textit{spaced $(d+1)$ edge coloring} if, whenever $e,e'\in E(G)$ with $c(e)=c(e')=d+1$, $e'$ is not $2$-reachable from $e$. \end{defi}

\begin{thm}\label{thm:spaced}If $G$ is a $d$-regular Vizing-class-2 graph that admits a spaced $(d+1)$-edge-coloring, then $\eta_p'(G)\leq 2d=2\chi'(G)-2.$
\end{thm}
\begin{proof} Let $c:E\rightarrow \{1,2,\ldots,d+1\}$ be a spaced $(d+1)$-edge-coloring of $G$.  Let $A=\{1,2,\ldots,d\}\cup \{\frac{1}{2}\}$. Define  $f:E(G)\rightarrow A$ so that $c^{-1}(i)=f^{-1}(i)$ for each $i\in \{1,2,\ldots,d\}$, and $c^{-1}(d+1)=f^{-1}(\frac{1}{2})$ (so $f$ recolors all of the edges colored $d+1$ with $\frac{1}{2}$, instead). For each $v\in V$, define $a_v$ to be the unique element of $A$ that does not appear in $f(\Gamma(v))$. Then, as in the preceding proofs,

\begin{align*}
\sum _{e\in N'(vw)}f(e)-\sum_{e\in N'(vx)}f(e)=& a_x-a_w-2(f(vw)-f(vx)).\\
\end{align*}

 If $a_w=a_x$, then $a_x-a_w-2(f(vw)-f(vx))=-2(f(vw)-f(vx))\neq 0$, because $f$ is a proper edge coloring. If $a_w\neq a_x$, then either $a_w\neq \frac{1}{2}$ or $a_x\neq \frac{1}{2}$, so $\frac{1}{2}$ appears in $f$ at $w$ or at $x$.  By the definition of spaced coloring, $\frac{1}{2}$ cannot appear at both $w$ and $x$ because $w,v,x$ is a path of length two between $w$ and $x$, and hence exactly one of $a_w$ or $a_x$ is $\frac{1}{2}$. But then $a_x-a_w$ is not an integer and $2(f(vw)-f(vx))$ is an integer, so we have that $a_x-a_w-2(f(vw)-f(vx))\neq 0$.

Now, define $g:E(G)\rightarrow 2A$ by $g=2f$. Then for any incident edges $vw$ and $vx$,

$$\sum _{e\in N'(vw)}g(e)-\sum_{e\in N'(vx)}g(e)=2\left(\sum _{e\in N'(vw)}f(e)-\sum_{e\in N'(vx)}f(e)\right)\neq 0,$$ and $g$ is a proper additive edge coloring of $G$ with colors from the set $\{1\}\cup \{2,4,6,\ldots,2d\}$. Thus, $\eta_p'(G)\leq 2d=2\chi'(G)-2$, as desired. 
\end{proof}

 The following definition of the resistance of a graph was given in \cite{MR1630478} and \cite{MR2043807} (called the ``color number" in \cite{MR1630478}) for $3$-regular Vizing-Class-2 graphs, but extends naturally to $d$-regular Vizing-class-2 graphs. 

\begin{defi} Given a $d$-regular Vizing-class-2 graph $G$, define the \textit{resistance} of $G$, denoted $\mathit{r(G)}$, to be the minimum possible size of a color class in a proper ($d+1$)-edge-coloring of $G$. Equivalently,  \[r(G)=\min\{|c^{-1}(d+1)|:c \text{ is a proper $(d+1)$-edge-coloring of } G  \}.\] \end{defi}

\begin{prop} \label{prop:rg1} Suppose $G$ is a $d$-regular Vizing-class-2 graph with $r(G)=1$. If $G$ admits a proper edge coloring $c:E(G)\rightarrow \{1,2,\dots,d+1\}$ such that $c^{-1}(d+1)=\{e\}$ for some edge  $e$ that is not contained in any $3$-cycle of $G$, then $\eta_p(G)\leq 2d$. \end{prop}

\begin{proof}
  Any such coloring $c$ is a spaced coloring since $e$ is $2$-reachable from itself only if it is in some $3$-cycle.  The result follows from Theorem \ref{thm:spaced}.  
\end{proof}

The \textit{girth} of a graph ${G}$ is the least integer $k$ such that $G$ contains a cycle of length $k$. By Proposition \ref{prop:rg1}, if $G$ is $d$-regular, $r(G)=1$, and $\text{girth}(G)\geq 4$, then $\eta_p'(G)\leq 2d$.   Theorem \ref{thm:girth}, whose proof is the main topic of the remainder of this section, extends this idea to graphs with $r(G)\geq 2$ by showing that $\eta'_p(G)\leq 2d$ for graphs whose girth is large enough relative to $r(G)$.

\begin{thm}\label{thm:girth} If $G$ is a $d$-regular Vizing-class-2 graph with resistance $r\geq 2$ and $\text{girth}(G)\geq 6r-5$, then $\eta_p'(G)\leq 2d=2\chi'(G)-2$.
\end{thm}
We will use the following lemma in the proof of Theorem \ref{thm:girth}. 

\begin{lem} \label{lem:girthenough}
   Let $\ell\geq 2$ be an integer, and and let $G$ be a graph with $\text{girth}(G)\geq 6\ell-5$. If $C$ is a cycle in $G$ and $S\subseteq E(G)$ is such that $|S|=\ell-1$ and $S\cap E(C)=\emptyset$, then there is some edge $e$ of $C$ which is not $2$-reachable from any edge in $S$.  
\end{lem}

\begin{proof} 

For the sake of contradiction, suppose that there is some least $\ell\geq 2$ for which there is a counterexample $G$ with girth $g\geq 6\ell-5$. Then $G$ contains some cycle $C$ and some $S\subseteq E(G)$ with $|S|=\ell-1$ and $S\cap E(C)=\emptyset$, where every edge of $C$ is $2$-reachable by some edge in $S$.  

Pick $e^*\in S$ and define  \[R(e^*)= \{e^*\}\cup \{e\in E(G):e \text{ is $2$-reachable from } e^*\}.\]

Define $T(e^*)=G[(R(e^*)]$, the subgraph of $G$ induced by $R(e^*)$.  Note that  $R(e^*)$ is composed of the edge $e^*$ and the edges of all possible paths of length $\leq 3$ in $G$ that begin at either endpoint of $e^*$. Thus,  $T(e^*)$ is connected, a longest path in $T(e^*)$ has length at most $7$, and any path of length $7$ in $T(e^*)$ contains the edge $e^*$. Furthermore, if $T(e^*)$ has any cycle of length $\geq 7$, then the length of the cycle is exactly $7$ and the cycle must include the edge $e^*$ .

We will first show that it is not the case that $\ell=2$: If $\ell=2$, then $|S|=1$ and therefore $S=\{e^*\}$. Since $C$ is a cycle of length at least $ 6(2)-5=7$ that does not contain the edge $e^*$, the edges of $C$ cannot form a cycle in $ T(e^*)$ and there is some edge of $C$ that is not $2$-reachable from $e^*$. Thus, any least counterexample must have $\ell\geq 3$ and $g\geq 6\ell-5 \geq 13$. In this case,  $T(e^*)$ cannot contain any cycles and is therefore a tree. 

\begin{figure}[H]
    \centering

 \includegraphics[scale=.45]{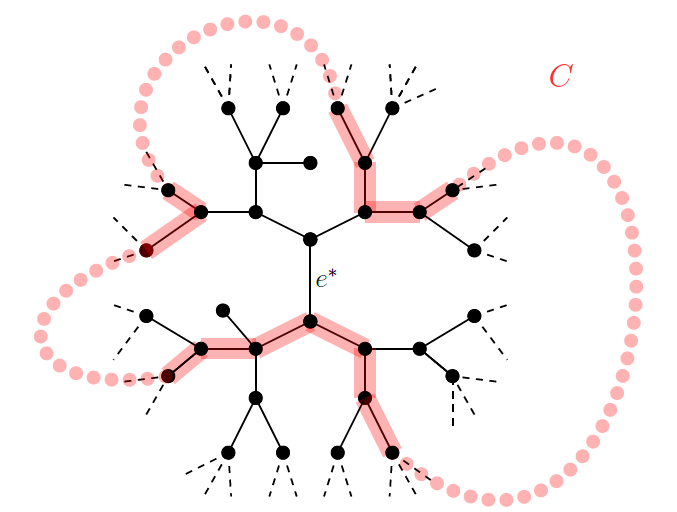}
\caption{A depiction of the tree $T(e^*)$ and the cycle $C$. The highlighted edges $E(C)\cap E(T(e^*))$  are contracted to form $G'$. }
    \label{fig:cycle C}
\end{figure}

 Form a new graph $G'$ from $G$ by contracting the edges in $E(T(e^*))\cap E(C)$.  Denote by $C'$ the cycle that results from $C$ after contracting these edges. By assumption, every edge of $C$ that is not $2$-reachable from $e^*$ in $G$ must be $2$-reachable from some $e\in S\setminus \{e^*\}$ in $G$, and since contracting edges cannot increase the distance between any pair of non-contracted edges,
 every edge in $C'$ is  $2$-reachable in $G'$ by some $e\in S'=S\setminus \{e^*\}$. Note also that $S'\subseteq E(G')$ because none of the edges in $S'$ are contracted to form $G'$, and note that $|S'|=\ell-2$.

 We show that $\text{girth}(G')\geq 6(\ell-1)-5$. Since $\text{girth}(G) \geq 6\ell-5$, if there is some cycle $D'$ of $G'$ whose length is less than $6(\ell-1)-5$, $D'$ must be the result of contracting more than $6$ edges of some cycle $D$ of $G$. The intersection of the set of edges of a tree with the set of edges of a cycle always forms a (possibly empty) set of pairwise disjoint subpaths of the cycle, and thus $E(T(e^*))\cap E(D)$ can be written as $E(P_1)\cup E(P_2)\cup\cdots\cup E(P_k)$, where $P_i$ is a subpath of $D$ for each $i\in\{1,2,\dots,k\}$ and $V(P_i)\cap V(P_j)=\emptyset$ for every pair $i,j\in \{1,2,\dots,k\}$ with $i\neq j$. The edges in $E(D)\setminus E(T(e^*))$ therefore also form a set of $k$ pairwise disjoint subpaths of $D$; $E(D)\setminus E(T(e^*))=E(Q_1)\cup E(Q_2)\cup \cdots \cup E(Q_k)$,  where the cycle $D$ is given by $P_1,Q_1,P_2,Q_2,\dots,P_k,Q_k$ and for each $i\in\{1,2,\dots,k\}$, $Q_i$ is a subpath of $D$ that does not contain any edges of $T(e^*)$. Note that both endpoints of each $Q_i$ are in $V(T(e^*))$.
 
\begin{figure}[H]
    \centering
\includegraphics[scale=.47]{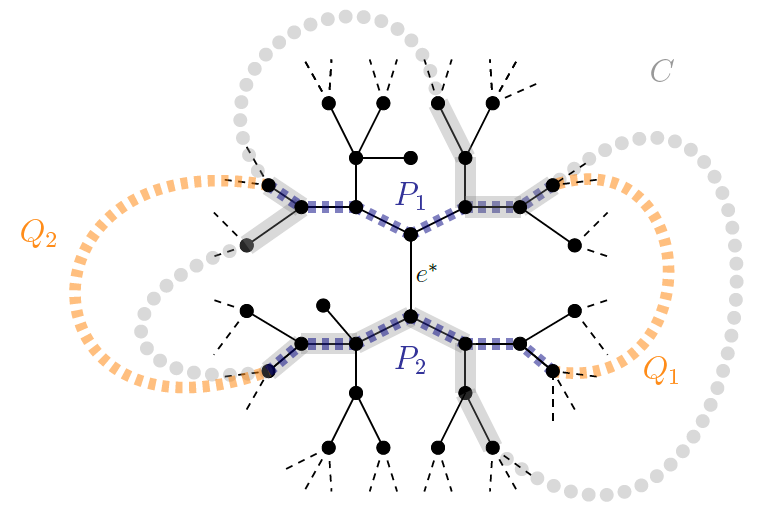}
\caption{A depiction of a cycle $D$ given by the paths $P_1,Q_1,P_2,Q_2$. At most $6$ edges in each $P_i$ are contracted to form the cycle $D'$ in $G'$ and each $Q_i$ has length at least $6(\ell-2)$. }
    \label{fig:cycle D}
\end{figure}

  The edges of $D$ that are contracted to form $D'$ each lie on some $P_i$, and because we do not contract $e^*$ when we form $G'$, at most $6$ edges in each $P_i$ are contracted to obtain $D'$. Therefore, since we have established that more than $6$ edges of $D$ are contracted to form $D'$, it must be the case that $k\geq 2$; $D$ contains at least two maximal vertex-disjoint subpaths of $T(e^*)$.

Fix $i\in \{1,2,\dots,k\}$, and let $x$ and $y$ denote the endpoints of $Q_i$. Then $x,y\in V(T(e^*))$, so there is a unique path $P$ in $T(e^*)$ from $x$ to $y$. Thus, $P$ and $Q_i$ together form a cycle in $G$, and since $P$ has length at most $7$ and a smallest cycle in $G$ has length at most $6\ell-5$, the length of $Q_i$ is at least $6\ell-5-7=6(\ell-2)$. Since none of the edges of any $Q_i$ are contracted when we form $G'$, $D'$ has length at least $k\cdot 6(\ell-2)\geq 2\cdot 6(\ell-2) =12(\ell-2)$. Since $\ell \geq 3$, $12(\ell-2)\geq 6(\ell-1)-5$, and  $D'$ has length at least $ 6(\ell-1)-5$. Thus, there are no cycles $D'$ of length less than $6(\ell-1)-5$ in $G'$, and we conclude that $\text{girth}(G')\geq 6(\ell-1)-5$.

Therefore, $G'$ is a graph with girth at least $6(\ell-1)-5$ that contains a cycle $C'$ and  set of edges $S'=S\setminus \{e^*\}\subseteq E(G')$ such that $E(C')\cap E(S')= \emptyset$, $|S'|=(\ell-1)-1$, and every edge of $C'$ is $2$-reachable from some $e\in E(S')$. This contradicts the minimality of $\ell$. 

\end{proof}

  In the following proof of Theorem \ref{thm:girth}, it will be useful to distinguish between the labels of different edges that are colored $d+1$. To do this, we will consider colorings of the form \[f:E\rightarrow \{1,2,\dots,d,(d+1,1),(d+1,2),\dots,(d+1,\ell)\}\] and let $\widetilde{f}$ denote the $(d+1)$-edge-coloring obtained by replacing each edge-label of the form $(d+1,i)$ in $f$ with $d+1$ (making the $(d+1)$-edge-labels indistinguishable again in $\widetilde{f}$). 

\begin{proof}[Proof of Theorem \ref{thm:girth}]
Suppose $G=(V,E)$ is a $d$-regular Vizing-class-2 graph with $r(G)=r\geq 2$ and $g\geq 6r-5$. Let $c:E(G)\rightarrow \{1,2,\dots, d+1\}$ be a proper edge coloring of $G$ with $|c^{-1}(d+1)|=r$. We give a recoloring procedure that produces a spaced $(d+1)$-coloring of $G$ so that Theorem \ref{thm:spaced} applies.

Let $e_1,e_2,\dots, e_r$ be the $r$ edges that are  colored $d+1$ in $c$ and define the labeling \[c_0:E\rightarrow \{1,2,\dots,d,(d+1,1),(d+1,2),\dots,(d+1,r)\}\] by \[c_0(e)= \begin{cases}
    c(e)\quad& \text{ if } e\in c^{-1}(\{1,2,\dots,d\}) \\
    (d+1,i)\quad& \text{ if } e=e_i \text{ for some } i\in\{1,2,\dots,r\}
\end{cases} \] so that $c_0$ distinguishes the $r$ labels of the edges labeled with the color $d+1$, and $\widetilde{c_0}=c$. Our general recoloring procedure is as follows: for each $i\in\{1,2,\dots,r\}$, we will  (recursively) construct a proper edge coloring $c_i:E\rightarrow \{1,2,\dots,d,(d+1,1),(d+1,2),\dots,(d+1,r)\}$ such that $c_i$ satisfies the following property, $P_i$. \vspace{.3cm} 

\noindent \textit{Property $P_i$}: $\widetilde{c_i}$ is a proper $(d+1)$-edge-coloring of $G$ and if $e,e'\in E$ with $c_i(e)=(d+1,j)$ and $c_i(e')=(d+1,k)$ and $j,k\leq i$, then $e'$ is not $2$-reachable from $e$.  \vspace{.3cm}

Observe that if $c_r$ satisfies property $P_r$,  $\widetilde{c_r}$ will be a \textit{spaced} $(d+1)$-edge-coloring of $G$, so Theorem \ref{thm:spaced} will apply.

For the $i=1$ case, setting $c_1=c_0$, we see that $c_1$ satisfies property $P_1$ because $\widetilde{c_0}$ is a proper $(d+1)$-coloring of $G$ and the existence of a path of length two between the endpoints of the single edge whose label is $(d+1,1)$ would imply that $G$ contains a $3$-cycle, contradicting the fact that $\text{girth}(G)\geq 6r-5\geq 7$.

Now suppose that $2\leq i<r$ and that $c_{i-1}:E\rightarrow \{1,2,\dots,d,(d+1,1),(d+1,2),\dots,(d+1,r)\}$ satisfies property $P_{i-1}$. If $c_{i-1}$ also satisfies property $P_i$, then set $c_i=c_{i-1}$ and we are done. If not, we give a procedure that moves the label $(d+1,i)$ along the edges of a walk $W$ in $G$ and recolors the other edges on $W$ appropriately so that we eventually end up with a labeling that satisfies property $P_i$. We will let $W_{j}=x_0,x_1,\dots,x_j,x_{j+1}$ denote the first $j$ steps of this walk, produce a labeling $c_{i-1,j}:E\rightarrow \{1,2,\dots,d,(d+1,1),(d+1,2),\dots,(d+1,r)\}$ for each $j$, and show that eventually, for some $q\in \mathbb{N}$, $c_{i-1,q}$ must satisfies property $P_i$. (Note that Example \ref{exam:recoloring} also illustrates the recoloring process.)

 Let $h$ denote the (unique) edge that is colored $(d+1,i)$ in $c_{i-1}$, and set $x_0$ and $x_1$ to be the two endpoints of $h$, so that $W_0=x_0,x_1$. Because $G$ is $d$-regular and $\widetilde{c_{i-1}}$ is a $(d+1)$-edge-coloring with
 $\widetilde{c_{i-1}}(e_i)=d+1$, some color $a_{x_0}\in \{1,2,\dots,d\}$ is missing at $x_0$ in $c_{i-1}$. Note that $a_{x_0}$ cannot also be missing at $x_1$, since if that were the case, we could change the color of $x_0x_1$ in $c_{i-1}$ to $a_{x_0}$ to get a proper coloring with fewer than $r$ edges colored $d+1$, contradicting the minimality of $r$. Hence, there is an edge $x_1x_2$, with $x_2\neq x_0$, such that $c_{i-1}(x_1x_2)=a_{x_0}$. Recolor the edges with $c_{i-1,1}$ defined by \[ c_{i-1,1}(e)=
 \begin{cases} c_{i-1}(e) \quad & \text{ if } e\notin \{x_0x_1,x_1x_2\}\\
a_{x_0} \quad &  \text{ if } e= x_0x_1\\
(d+1,i) \quad &  \text{ if } e= x_1x_2\\
 \end{cases}\]

 and add $x_2$ to the walk, so that $W_1=x_0,x_1,x_2$. Since we recolored $x_0x_1$ with $a_{x_0}$, no color appears twice at $x_0$. However, it is not necessarily the case that $\widetilde{c_{i-1,1}}$ is a proper $d+1$ coloring of $G$, since there could exist some $y\in V$ and some $j\neq i$ such  that $c_{i-1,1}(x_2y)=(d+1,j)$. This will not cause issues; we will continue moving the label $(d+1,i)$ along our walk until property $P_i$ is eventually  satisfied. We take the time to construct $c_{i-1,2}$ before giving the recursive construction of $c_{i-1,j}$  for $j\geq 3$ because the $c_{i-1,2}$ step demonstrates all of the nuances involved in the subsequent steps.

 Because $a_{x_0}$ and $(d+1,i)$ now both appear at $x_1$, it must be the case that there is some $a_{x_1}\in\{1,2,\dots,d\}\setminus\{a_{x_0}\}$ that does not appear at $x_1$.  Note that $a_{x_1}$ must appear at $x_2$, otherwise changing the color of $x_1x_2$ in $c_{i-1,1}$ to $a_{x_1}$ would give a proper coloring with fewer than $r$ edges colored $d+1$. Hence, there is an edge $x_2x_3$, with $x_3\neq x_1$, such that $c_{i-1,1}(x_2x_3)=a_{x_1}$.  Recolor the edges with $c_{i-1,2}$ defined by

\[ c_{i-1,2}(e)=
\begin{cases} c_{i-1,1}(e) \quad & \text{ if } e\notin \{x_1x_2,x_2x_3\}\\
a_{x_1} \quad &  \text{ if } e= x_1x_2\\
(d+1,i) \quad &  \text{ if } e= x_2x_3\\
 \end{cases}\] 

and set $W_2=x_1,x_2,x_3$.

Since $x_1x_2$ is colored with $a_{x_1}$ and $a_{x_1}$ was not already present at $x_1$ in $ c_{i-1,2}$, no color now appears more than once at $x_1$. It could be at this point that $\widetilde{c_{i-1,2}}$ is not a proper $(d+1)$-edge-coloring of $G$, since there could exist some $y\in V$ and $j\neq i$ such that $c_{i-1,2}(x_2y)=(d+1,j)$ or $c_{i-1,2}(x_3y)=(d+1,j)$ with $j\neq i$. Note, however, that every vertex except possibly $x_2$ and $x_3$ has $d$ distinct colors on its incident edges in $\widetilde{c_{i-1,2}}$, and $d+1$ is the only color that might possibly be repeated at $x_2$ or $x_3$ and it can appear at most twice at each of $x_2$ and $x_3$.

Continue recursively with this same process, defining $c_{i-1,j}$  for each $j\geq 3$ by selecting $a_{x_{j-1}}\in \{1,2,\dots,d\}\setminus\{a_{x_{j-2}}\}$ so that $a_{x_{j-1}}$ does not appear at vertex $x_{j-1}$ in $c_{i-1,j-1}$ and selecting $x_{j+1}\in N(x_j)$ so that $x_{j+1}\neq x_{j-1}$ and $c_{i-1,j-1}(x_jx_{j+1})=a_{x_{j-1}}$. (Such an $a_{x_{j-1}}$ must exist since $(d+1,i)$ and $a_{x_{j-2}}$ are both present at $a_{x_{j-1}}$, and such an $x_{j+1}$ must always exist by minimality of $r$). Then define $c_{i-1,j}$ by

\[ c_{i-1,j}(e)=
 \begin{cases} c_{i-1,j-1}(e) \quad & \text{ if } e\notin \{x_{j-1}x_{j},x_jx_{j+1}\}\\
a_{x_{j-1}} \quad &  \text{ if } e= x_{j-1}x_j\\
(d+1,i) \quad &  \text{ if } e= x_jx_{j+1}\\
 \end{cases}.\]
 
and set $W_{j}=x_1,x_2,\dots,x_{j+1}$. By the same reasoning as in the discussion about $c_{i-1,2}$, every vertex except possibly $x_j$ and $x_{j+1}$ has $d$ distinct colors on its incident edges in $\widetilde{c_{i-1,j}}$, and $d+1$ is the only color that might possibly be repeated at $x_j$ or $x_{j+1}$, and it can appear at most twice at each of $x_j$ and $x_{j+1}$. 
 
In each walk $W_j$, the labeling $c_{i-1,j}$ colors the edge $x_jx_{j+1}$ with $(d+1,i)$ and all other edges in $W_j$ with an element of $\{1,2,\dots,d\}$.  Furthermore, for each walk $W_j$,  $x_{k-1}\neq x_{k+1}$ for all $k\in\{1,2,\dots,j\}$, and hence  $x_{k-1}x_{k}\neq x_{k}x_{k+1}$ for each $k$, and $W_j$ never ``backtracks."  Thus, since $G$ is finite,  there is some least $m$ such that $x_{m+1}=x_\ell$ for some  $ \ell <m+1$, and $C$ given by $x_\ell,x_{\ell+1},\dots, x_{m-1},x_m,x_{m+1}=x_{\ell}$ is a cycle in $G$. By our assumption on the girth of $G$, $C$ has length at least $6r-5$.

In the recoloring process, $(d+1,i)$ travels along the walk $W_m$, always moving to an edge previously colored with a color in $\{1,2,\dots,d\}$ and its previous edge is always colored with a color in $\{1,2,\dots,d\}$, so for each $j\in\{\ell,\ell+1,\dots,m-1,m\}$, $c_{i-1,j}$ assigns the label $(d+1,i)$ to exactly one edge of $C$ and all other edges of $C$ are assigned labels from $\{1,2,\dots,d\}$.

Let $S=\{e\in E: c_{i-1}(e)=(d+1,n) \text{ for some } n\neq i\}$. Then $|S|=r-1$, and since $G$ has girth at least $6r-5$, there exists some edge $x_{q}x_{q+1}$ in $C$ that is not $2$-reachable from any edge in $S$, by Lemma \ref{lem:girthenough}. This also guarantees that no edges in $S$ are incident to $x_qx_{q+1}$ (since such an edge would be $2$-reachable), and thus $\widetilde{c_{i-1,q}}$ is a proper $(d+1)$-edge-coloring of $G$. For each $e$ not on $W_q$, $c_{i-1,q}(e)=c_{i-1}(e)$ , so $(d+1,k)$ is not $2$-reachable from $(d+1,j)$ for any $j,k<i$. Hence, setting $c_{i}=c_{i-1,q}$, $c_i$ satisfies property $P_i$, as desired. 

Finally, since $c_r$ satisfies property $P_r$, $\widetilde{c_r}$ is a spaced $(d+1)$-edge-coloring of $G$, and by Theorem \ref{thm:spaced}, $\eta'(G)\leq 2d$.

\end{proof}
\newpage

\begin{exam}\label{exam:recoloring} Even though the Petersen graph does not have large enough girth relative to its resistance to satisfy the hypothesis of Theorem \ref{thm:girth}, we use it to demonstrate the recoloring procedure in the preceding proof (this is easier than dealing with a large graph). Let $G$ denote the Petersen graph, as shown below. Starting with a proper edge coloring $c:E(G)\rightarrow \{1,2,3,4\}$ with $|c^{-1}(4)|=2=r(G)$, we exhibit the colorings $c_1$, $c_{1,1}$, $c_{1,2}$, $c_{1,3}$ and the corresponding walks $W_0,W_1,W_2,$ and $W_3$ to illustrate the process described in the preceding proof.

\begin{figure}[htbp]
    \centering
\includegraphics[scale=.55]{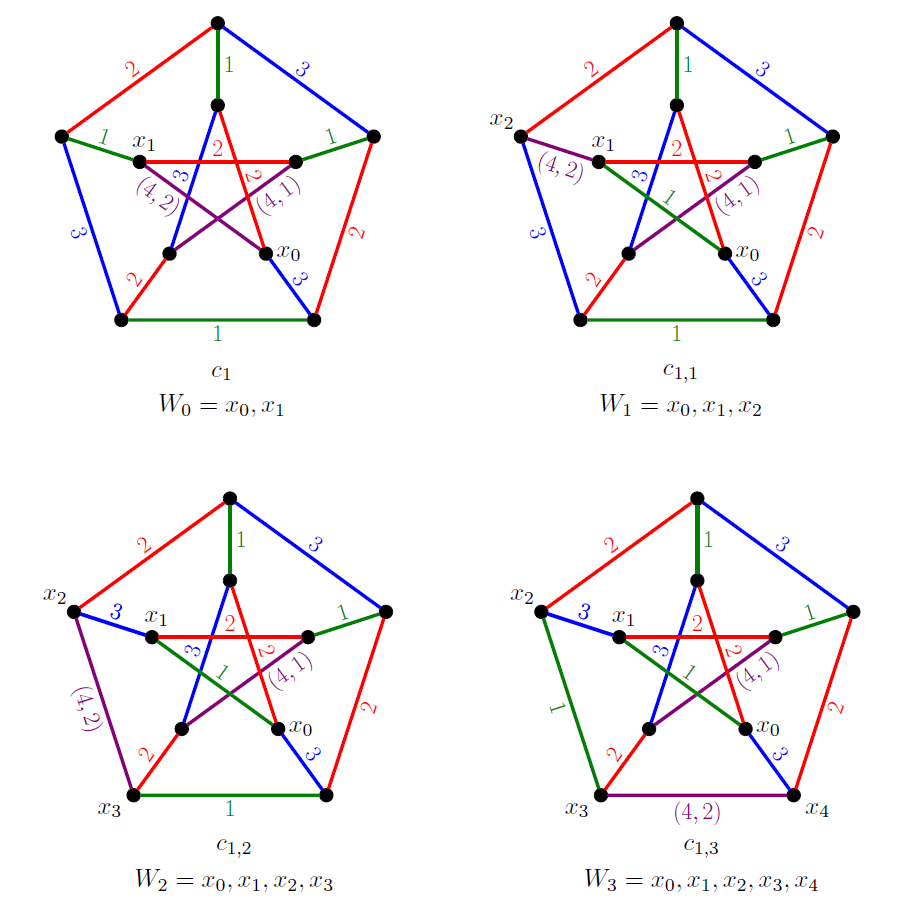} 

\caption{The recoloring procedure from Theorem \ref{thm:girth}.}
    \label{fig:recoloring}
\end{figure}
The edge label $(4,2)$ travels along a walk that eventually must contain a cycle, and that cycle will not contain the edge labeled $(4,1)$. In graphs where the girth and resistance restrictions of Theorem \ref{thm:girth} are met (the restrictions are not met in this example), such a cycle must contain an edge that is not $2$-reachable from the edge labeled $(4,1)$ and the recoloring that assigns the label $(4,2)$ to this edge will satisfy property $P_2.$

\end{exam}

In  \cite{MR2097332}, it was shown that for each $d\geq 3$ and each $g\in \mathbb{N}$, the probability that a random $d$-regular graph on $n$ vertices has girth at least $g$ tends to a positive constant as $n\rightarrow \infty$. Furthermore, if $G$ is a $d$-regular graph with $d$ even and $|V(G)|$ odd, then $G$ is necessarily Vizing-Class-2 because $d$ matchings, each of size at most $\lfloor \frac{|V(G)|}{2}\rfloor<\frac{|V(G)|}{2}$, cannot cover all of the $\frac{d|V(G)|}{2}$ edges of $G$. Thus,  for each even $d\geq 2$, there are $d$-regular Vizing-Class-2 graphs with arbitrarily large girth. To use Theorem \ref{thm:girth}, however, we need graphs that not only have large girth, but also small resistance. The following proposition gives infinite classes of graphs to which Theorem \ref{thm:girth} applies. 

\begin{prop}  Let $d\geq 4$ be an even integer. For any  $g\in \mathbb{N}$, there exists a $d$-regular graph with resistance $\frac{d}{2}$ and girth at least $g$.
\end{prop}
\begin{proof}
   Let $d\geq 4$ be even, and $g\in \mathbb{N}$. In \cite{MR2097332}, it was shown that the probability that a random $d$-regular bipartite graph $G$ has girth greater than $g$ tends to a positive constant as $|V(G)|\rightarrow \infty$.   Thus, there exists a $d$-regular bipartite graph $H$ with $\text{girth}(H)\geq g$. (See also the explicit construction in the case that $d$ is a power of $2$ in \cite{MR1339092}). Since bipartite graphs are Vizing-Class-1 (\cite{MR1511872}), there is a proper coloring $c:E(H)\rightarrow \{1,2,\dots,d\}$ of $H$. 

Subdivide an edge $xy$ of $H$, replacing $xy$ with edges $xv$ and $vy$, where $v$ is a new vertex, and  call the resulting graph $H'$. Then $H'$ has an odd number of vertices and any color class in an edge coloring of $H'$ contains at most $\lfloor\frac{|V(H')|}{2}\rfloor=\lfloor\frac{|V(H)|+1}{2}\rfloor=\frac{|V(H)|}{2}$ edges. Thus, $d$ color classes contain at most $d\frac{|V(H)|}{2}=|E(H')|-1$ edges, and $H'$ must be Vizing-Class-2.  Let $H'_1,H'_2,\dots,H'_{\frac{d}{2}}$  be $\frac{d}{2}$ disjoint copies of $H'$, and for each $i\in\{1,2,\dots,\frac{d}{2}\}$ let $x_i, y_i,$ and $v_i$ denote the copies of the vertices $x,y,$ and $v$ in $H_i$, respectively. 

\begin{figure}[H]
    \centering
\includegraphics[scale=.43]{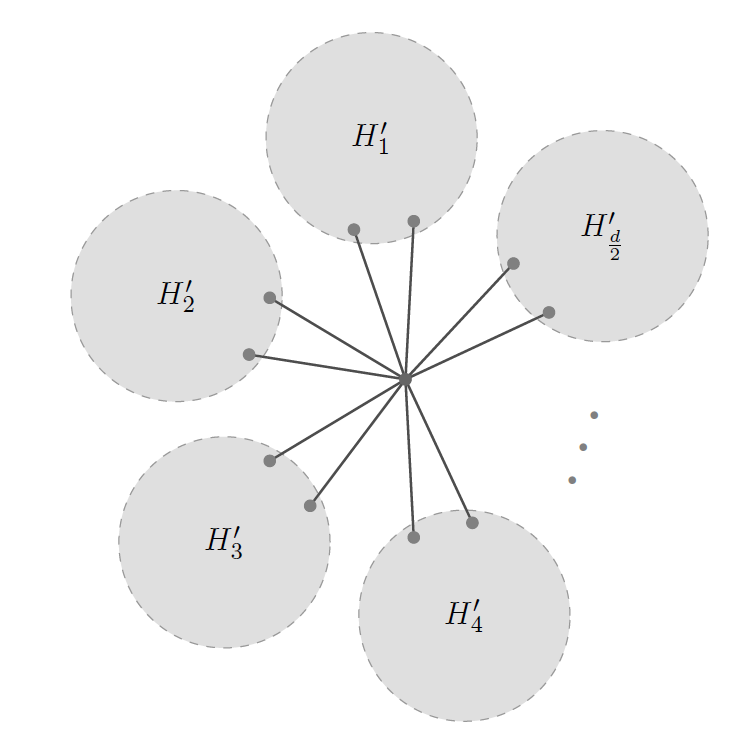}

\caption{The construction of a $d$-regular graph $G$ with large girth and $r(G)=\frac{d}{2}$.}
    \label{fig:subdivide}
\end{figure}

Construct a new $d$-regular graph $G$ by identifying the vertices $v_i$, $i=1,2,\dots,\frac{d}{2}$. This identification does not create any new cycles, so  $\text{girth}(G)\geq g$. Since any $d+1$-edge-coloring of $G$ must have at least one edge of each $H_i'$ colored with $d+1$, we see that $r(G)\geq \frac{d}{2}$. We exhibit a proper $d+1$-edge-coloring of $G$ with $c^{-1}(d+1)=\frac{d}{2}$.  First, properly color each $H_i'$ so that the edges not incident to $v_i$ are colored the same as they are colored in $c:E(H)\rightarrow \{1,2,\dots, d\}$, and color $x_iv_i$ with $c(xy)$ and $v_iy_i$ with $d+1$. Then there is some color $j\neq c(xy)$ that appears on an edge $y_iw_i$ in $H_i'$, where $w_i\in N_{H_i'}(y_i)$ and $w_i\neq x_i$. Swapping the colors of $v_iy_i$ and $y_iw_i$ still yields a proper $d+1$-edge-coloring of each $H_i'$, with the property that $d+1$ no longer appears at any of the vertices $v_i$ that were added in the subdivision step. Finally, for each $H_i'$, permute the color classes  (while leaving the $d+1$ color class fixed) so that the two colors incident to $v_i$ are the colors $2i-1$ and $2i$. Coloring the edges of each $H_i'$ in this way produces a proper $(d+1)$-edge-coloring of $G$ with $r(G)=\frac{d}{2}$ edges colored with $d+1$.

\end{proof}

\section{Concluding Remarks}
The reader may have noticed that the proofs of Theorems \ref{thm:class1add} and \ref{thm:addbound} differ from those of the other results in this work in that they provide a set $A\subseteq \mathbb{N}$ such that \textit{every} proper edge coloring $c:E(G)\rightarrow A$ is an additive edge coloring, as opposed to \textit{at least one} such proper coloring existing. In light of this, define the \textit{strong proper additive chromatic number} of a graph $G$, denoted $\eta_{sp}(G)$, to be the least $k\in \mathbb{N}$ such that there exists some $A\subseteq  \{1,2,\dots, k\}$ with the property that $G$ can be properly colored using the elements of $A$ and every proper coloring of $G$ using colors from $A$ is also an additive coloring. 

Define the \textit{strong proper additive chromatic index}, $\eta_{sp}'(G)$, analogously for edge colorings. Then Theorems \ref{thm:class1add} and \ref{thm:addbound} also hold true if we replace $\eta_p'(G)$ with $\eta_{sp}'(G)$ in their statements.  One course of future research could be to determine precise relationships between $\eta_p(G)$ and $\eta_{sp}(G)$ or between $\eta_p'(G)$ and $\eta_{sp}'(G)$ for other classes of graphs.  More specifically, one might consider the following question.

\begin{ques} \label{ques:p=sp}For which graphs G (other than line graphs of $d$-regular Vizing-class-1 graphs), does $\eta_p(G)=\eta_{sp}(G)$?
\end{ques}

Additionally, a question that arises naturally from the results in this work is the following.

\begin{ques}\label{ques:strengthening}
    Is it true that if $G$ is a $d$-regular graph, then $\eta_P'(G)=\chi'(G)$?
\end{ques}

 An affirmative answer to Question \ref{ques:strengthening} would be a strengthening of the theorems in this work and would also be a strengthening of the additive coloring conjecture for the class of line graphs of $d$-regular graphs.
 
 \section*{Acknowledgements} 
The author thanks the members of ACRONYM, Truman State University's advanced combinatorics research organization, for feedback on some of the details of the proofs in this work. In particular, thanks are due to Anastasia Halfpap, Wayne Johnson, and Stephen Lacina.

\bibliographystyle{plain}
\bibliography{references}
\vspace{1cm}

Department of Mathematics, Truman State University, Kirksville, Missouri\

igossett@truman.edu\

\end{document}